\documentclass{amsart}

\newcommand{\cC}{\mathcal{C}}
\newcommand{\cE}{\mathcal{E}}

\newcommand{\bQ}{\mathbb{Q}}\newcommand{\bR}{\mathbb{R}}
\newcommand{\bT}{\mathbb{T}}

\newcommand{\bZ}{\mathbb{Z}}

\usepackage[dvipsnames,svgnames,x11names,hyperref]{xcolor}
\usepackage{mathtools,url,graphicx,verbatim,amssymb,enumerate,stmaryrd}
\usepackage[pagebackref,colorlinks,citecolor=Bittersweet,linkcolor=Bittersweet,urlcolor=Bittersweet,filecolor=Bittersweet]{hyperref}
\usepackage{microtype,cite}
\usepackage[all,cmtip]{xy}
\usepackage[margin=1.45in]{geometry}
\usepackage{setspace}
\newtheorem{theorem}{Theorem}[section]
\newtheorem{lemma}[theorem]{Lemma}
\newtheorem{proposition}[theorem]{Proposition}
\newtheorem{corollary}[theorem]{Corollary}

\theoremstyle{definition}
\newtheorem{definition}[theorem]{Definition}
\newtheorem{example}[theorem]{Example}
\newtheorem{remark}[theorem]{Remark}
\newtheorem{notation}[theorem]{Notation}

\newcommand{\cat}[1]{\mathsf{#1}}
\newcommand{\mr}[1]{{\rm #1}}
\newcommand{\fS}{\mathfrak{S}}
\newcommand{\bfc}{\underline{\mathbf{c}}}
\newcommand{\bfC}{\underline{\mathbf{C}}}

\newcommand{\lra}{\longrightarrow}

\title{Homological stability for unlinked  circles in a 3-manifold}
\author{Alexander Kupers}
\thanks{Alexander Kupers was supported by a William R. Hewlett Stanford Graduate Fellowship, Department of Mathematics, Stanford University, and was partially supported by NSF grant DMS-1105058. Report number: CPH-SYM-DNRF92.}
\email{kupers@math.ku.dk}
\address{Institut for Matematiske Fag, K\o benhavns Universitet, Universitetsparken 5, 2100 K\o benhavn \O.}
\date{\today}

\begin{document}

\begin{abstract}We prove a homological stability theorem for unlinked circles in $3$-manifolds and give an application to certain groups of diffeomorphisms of 3-manifolds.\end{abstract}

\maketitle

\tableofcontents

\section{Introduction} We start by introducing the spaces of interest. Let $\mr{Emb}(-,-)$ be the space of embeddings and $\mr{Diff}(-)$ the topological group of diffeomorphisms, both in the $C^\infty$-topology.

\begin{definition}Let $M$ be a $3$-manifold, then we define \emph{the space of $k$ unlinked circles} to be
\[\mathcal C_k(M) \coloneqq \mr{Emb}^\mr{unl}\left(\bigsqcup_{k} S^1,M\right)/\,\mr{Diff}\left(\bigsqcup_k S^1\right)\]
where $\mr{Emb}^\mr{unl}(-,-)$ is the subspace of embeddings $\bigsqcup_{k} S^1 \hookrightarrow M$ that extend to an embedding $\bigsqcup_{k} D^2 \hookrightarrow M$.
\end{definition}

\begin{notation}We will denote a point in $\mathcal C_k(M)$ by $\bfc$ and a circle in $\bfc$ by $c$.\end{notation}

Suppose we are given a chart $\bR \times \bR \hookrightarrow \partial M$ and a collar neighborhood $\bR \times \bR \times [0,\infty) \hookrightarrow M$ for this chart. Then we can define a stabilization map
\[t \colon \cC_k(M) \lra \cC_{k+1}(M)\]
as follows. We remark that $\cC_k(-)$ is a continuous functor from the category of $3$-manifolds and embeddings to topological spaces. We pick an embedding $\rho \colon D^3 \cup M \hookrightarrow M$ such that $\rho|_{M}$ is isotopic to the identity and $\rho|_{D^3}$ is given $(x,y,z) \mapsto (x,y,z+2)$ with respect to the collar neighborhood $\bR^3 \times [0,\infty)$. We also fix a point $\bfc_0 \in \cC_1(D^3)$, and define
\[t(\bfc) \coloneqq \rho(\bfc_0 \sqcup \bfc) \colon \cC_k(M) \lra \cC_{k+1}(M)\]

Since $\cC_1(D^3)$ is path-connected and $\rho$ is unique up to isotopy, this is unique up to homotopy.

\begin{theorem}\label{thm.main} If $M$ is path-connected the map $t_* \colon H_*(\cC_k(M)) \to H_*(\cC_{k+1}(M))$ is always injective and an isomorphism for $* \leq \lfloor \frac{k-2}{2} \rfloor$.\end{theorem}

\begin{remark}This result is inspired by work of Palmer \cite{palmerthesis}. Consider a path-connected $m$-dimensional manifold $M$ and a compact $n$-dimensional manifold $N$ such that $m \geq 2n+3$. Palmer proved homological stability for unlinked embeddings of $\bigsqcup_k N$ into $M$ where the stabilization map adds an unlinked copy of the manifold near the boundary. Here an embedding $\bigsqcup_k N \hookrightarrow M$ is said to be \emph{unlinked} if there exists an embedding $\bigsqcup_k \bR^m \hookrightarrow M$ such that each $\bR^m$ contains a single of copy of $N$.\end{remark}

Theorem \ref{thm.main} prompts the question what the stable homology is. This is discussed in Section \ref{sec.stable}. We also give two corollaries of this result. The first corollary is homological stability for spaces of Euclidean unlinked circles in $D^3$. \emph{Euclidean circles} are subsets of $\bR^3$ given by the image of a standard circle $\{(x,y,z) \in \bR^3|\,\sqrt{x^2+y^2}=r \text{ and } z=0\}$ under an isometry of $\bR^3$ with the standard Euclidean metric.

\begin{definition}\label{defeuclcircl}The  \emph{space of $k$ unlinked Euclidean circles in $D^3$} is the subspace of $\cC_k(D^3)$ consisting of Euclidean circles. It is denoted $\cE_k(D^3)$.\end{definition}

Note that this is homotopy equivalent to the space of $k$ unlinked Euclidean circles in $\bR^3$. Our corollary then follows from Theorem 4.1 of \cite{hatcherbrendle}:

\begin{theorem}[Brendle-Hatcher] \label{thmhbrigidifying} The inclusion $\cE_k(D^3) \hookrightarrow \cC_k(D^3)$ is a homotopy equivalence.\end{theorem}

\begin{corollary}\label{thm.cor} The map $t_* \colon H_*(\cE_k(D^3)) \to H_*(\cE_{k+1}(D^3))$ is always injective and an isomorphism for $* \leq \lfloor \frac{k-2}{2} \rfloor$.\end{corollary}

\begin{proof}If we use a collar neighborhood and embedding $\rho$ compatible with the Euclidean structure on $D^3$, and an $x \in \cE_1(D^3) \subset \cC_1(D^3)$, then the stabilization map used above restricts to a stabilization map $t \colon \cE_k(D^3) \to \cE_{k+1}(D^3)$ which fits into a commutative diagram
	\[\xymatrix{  H_*(\cC_k(D^3)) \ar[r]^-{t_*} \ar[d]_\cong &  H_*(\cC_{k+1}(D^3)) \ar[d]^\cong \\
		H_*(\cE_k(D^3)) \ar[r]_-{t_*} & H_*(\cE_{k+1}(D^3))}\]
with the top map an isomorphism in the desired range by Theorem \ref{thm.main}.
\end{proof}

\begin{remark}Hatcher and Wahl proved homological stability for the sequence of groups $\pi_1(\cE_k(D^3))$ in Corollary 1.2 of \cite{hatcherwahl}. These groups are known as \emph{string motion groups}. Wilson proved representation stability for pure string motion groups \cite{wilsonstring, wilsonstring2}, which implies rational homological stability for string motion groups. She also proved vanishing of the rational cohomology in positive degrees. Griffin gave a method to compute the integral cohomology in \cite{griffindiagonal}. It is known that $\cE_k(D^3)$ is not an Eilenberg-MacLane space, because $\cE_k(D^3)$ is homotopy equivalent to a finite CW complex while $\pi_1(\cE_k(D^3))$ contains torsion, see the remarks preceding Theorem 3 of \cite{hatcherbrendle}.\end{remark}

\begin{remark}This paper supersedes an earlier one about homological stability for the spaces $\cE_k(\bR^3)$.
\end{remark}

The second corollary is homological stability for certain groups of diffeomorphisms of 3-manifolds. Consider $M \#_k D^2 \times S^1$ for $M$ path-connected with non-empty boundary $\partial M$, and fix a diffeomorphism $\sqcup_k \bT^2 \to \partial(M \#_k D^2 \times S^1) \setminus \partial M$. Let $\mr{Diff}_{\partial M}(M \#_k D^2 \times S^1,\fS_\partial)$ be the topological group of diffeomorphisms $f$ that fix $\partial M$ pointwise and such that there exists a permutation $\sigma$ of $\{1,\ldots,k\}$ so that $f \circ \psi|_{\bT^2_i} = \psi|_{\bT^2_{\sigma(i)}}$. The following is proven in Section \ref{sec.generalizations}.

\begin{corollary}\label{corhomstabdiff} There is a map
	\[(t')_* \colon H_*(B\mr{Diff}_{\partial M}(M \#_k D^2 \times S^1,\fS_\partial)) \lra H_*(B\mr{Diff}_{\partial M}(M \#_{k+1} D^2 \times S^1,\fS_\partial))\]
which is an isomorphism for $* < \lfloor \frac{k-4}{2} \rfloor$ and a surjection for $* \leq \lfloor \frac{k-2}{2} \rfloor$.\end{corollary}

\subsection{Acknowledgments} We thank S\o ren Galatius, Allen Hatcher, Martin Palmer and the referee for helpful comments and conversations.

\section{The stable homology} \label{sec.stable} Whenever one proves a homological stability result, one next attempts to compute the stable homology. There are two obvious candidates for the stable homology, given by maps
\[C_k(\bR^3;\bR P^2) \lra \cC_k(\bR^3) \lra \Omega^3_0 \mr{Thom}(T\bR P^2)\] 
In this subsection we show that neither of these maps induces an isomorphism on $H_1$ for $k \gg 0$. To do so, we start with computing $H_1(\cC_k(\bR^3))$.

\begin{lemma}For $k \geq 2$ we have that $H_1(\cC_k(\bR^3)) \cong (\bZ/2\bZ)^3$.\end{lemma}

\begin{proof}By Theorem \ref{thmhbrigidifying} it suffices to compute $H_1(\cE_k(\bR^3))$. The group $H_1(\cE_k(\bR^3))$ can be computed by abelianizing the presentation of $\pi_1(\cE_k(\bR^3))$ in Proposition 3.7 of \cite{hatcherbrendle}.\end{proof}

First, we describe the map $C_k(\bR^3;\bR P^2) \to \cC_k(\bR^3)$. Let $C_k(\bR^3;\bR P^2)$ be the configuration space of $k$ unordered distinct points labeled by lines in $\bR P^2$:
\[C_k(\bR^3;\bR P^2) = \{((x_1,L_1),\ldots,(x_k,L_k)) \in (\bR^3 \times \bR P^2)^k | \, x_i \neq x_j \text{ for $i \neq j$}\}/\fS_k\]
The map $C_k(\bR^3;\bR P^2) \to \cC_k(\bR^3)$ sends $\{(x_i,L_i)\}_{i=1}^k$ to the configuration of $k$ circles $\bfc = \{c_i\}$ with $c_i$ the circle in the plane through $x_i$ orthogonal to $L_i$, with center $x_i$ and radius $\frac{1}{3} \min_{i \neq j}(||x_i-x_j||)$. The following implies that this map does \emph{not} induce an isomorphism on $H_1$ for $k \gg 0$.

\begin{lemma}For $k \geq 2$ we have that $H_1(C_k(\bR^3;\bR P^2)) \cong (\bZ/2\bZ)^2$.\end{lemma}

\begin{proof}Theorem A of \cite{RW} says that $H_1(C_2(\bR^3;\bR P^2)) \cong H_1(C_k(\bR^3;\bR P^2))$ for all $k \geq 2$. Let $F_2(\bR^3;\bR P^2)$ denote the labeled configuration space, defined as \[F_2(\bR^3;\bR P^2) \coloneqq \{((x_1,L_1),(x_2,L_2)) \in (\bR^3 \times \bR P^2)^2 | \, x_1 \neq x_2\}\]
There are two fiber sequences
\[\bR P^2 \times \bR P^2 \to F_2(\bR^3;\bR P^2) \to S^2\]
\[F_2(\bR^3;\bR P^2) \to C_2(\bR^3;\bR P^2) \to B\mathfrak S_2\]
The first has a section, so that $\pi_1(F_2(\bR^3;\bR P^2)) \cong \bZ/2\bZ^2$. Then the second implies that $\pi_1(C_2(\bR^3;\bR P^2)) \cong \bZ/2\bZ \wr \bZ/2\bZ$ with abelianization $(\bZ/2\bZ)^2$, so that $H_1(C_2(\bR^3;\bR P^2)) \cong (\bZ/2\bZ)^2$.\end{proof}

Secondly, the map $\cC_k(\bR^3) \to \Omega^3 \mr{Thom}(T\bR P^2)$ is the restriction of the scanning map defined in \cite{gmtw} to configurations of $k$ unlinked circles in $\bR^3$. We sketch its definition. The Thom space $\mr{Thom}(T\bR P^2)$ is the one-point compactification of the total space of $T\bR P^2$ and can be thought of as the space of affine lines in $\bR^3$, including a single line at infinity. We define
\[s \colon \cC_k(\bR^3) \lra \Omega^3 \mr{Thom}(T\bR P^2)\]
by giving its adjoint $S^3 \wedge (\cC_k(\bR^3))_+ \to \mr{Thom}(T\bR P^2)$. We think of $S^3$ as $\bR^3 \cup \{\infty\}$. For each configuration of unlinked  circles $\bfc$ in $\bR^3$ there exists an $\epsilon > 0$ such that for all $x \in \bR^3$ the ball $B_\epsilon(x) \cap \bfc$ is either  empty or contains a unique point closest to $x$. This $\epsilon$ gives an identification $B_\epsilon(x) \cong \bR^3$ and can be chosen continuously in $\bfc$.  Then a pair $(x,\bfc)$ is sent to either (i) the point at infinity if $B_{\epsilon}(x) \cap \bfc = \varnothing$ or (ii) to the affine line tangent to the circle at the unique closest point in $B_{\epsilon}(x) \cap \bfc$, using the identification $B_\epsilon(x) \cong \bR^3$. The point at infinity in $S^3$ is sent to the point at infinity in $\mr{Thom}(T\bR P^2)$.

Since $\cC_k(\bR^3)$ is path-connected, the scanning map lands in a single path component of $\Omega^3 \mr{Thom}(T\bR P^2)$ of $\Omega^3 \mr{Thom}(T\bR P^2)$, without loss of generality the $0$th one as all are weakly equivalent. The following implies that $s$ does \emph{not} induce an isomorphism on $H_1$ for $k \gg 0$.

\begin{lemma}We have that $H_1(\Omega^3_0 \mr{Thom}(T\bR P^2);\bQ) \cong \bQ$.\end{lemma}

\begin{proof}It suffices to show $\pi_4(\mr{Thom}(T\bR P^2)) \otimes \bQ \cong \bQ$. This will follow from the rational Hurewicz theorem and the following computation. Since $T\bR P^2$ is not orientable, there is a Thom isomorphism $\tilde{H}_{*+2}(\mr{Thom}(T\bR P^2);\bQ) \cong H_*(\bR P^2;\bQ_\omega)$ with $\bQ_\omega$ the non-trivial orientation local system. Poincare duality says $H_*(\bR P^2;\bQ_\omega) \cong H^{2-*}(\bR P^2;\bQ)$, so that
\[\tilde{H}_{*}(\mr{Thom}(T\bR P^2)) \cong \begin{cases} 0 & \text{if $* \neq 4$} \\
\bQ & \text{if $*=4$}
\end{cases}\]\end{proof}

We end with a positive result.

\begin{lemma}The stable homology of $\cC_k(\bR^3)$ is equal to that of a $3$-fold loop space.\end{lemma}

\begin{proof}Note that $\bigsqcup_k \cC_k(D^3)$ is an $E_3$-algebra. This implies that its stable homology is isomorphic to that of the $0$th component of the group completion $\Omega B(\bigsqcup_k \cC_k(D^3))$. This is a group-like $E_3$-algebra, and by the recognition principle weakly equivalent to a 3-fold loop space \cite{M}. \end{proof}

\begin{remark}The map $\Omega B(\bigsqcup_k \cC_k(D^3))\to \Omega^\infty \mr{Thom}(T\bR P^2))$ induced by the scanning map is homotopic to a 3-fold loop map.\end{remark}

\section{Simplicial techniques} \label{sec.simpl} To prove Theorem \ref{thm.resconn}, we will need a number simplicial techniques. 

\subsection{Definitions} First we recall some definitions.

\begin{definition} Let $\Delta_\mr{inj}$ be the category with objects the finite non-empty sets $[p] = \{0,\ldots,p\}$ for $p \geq 0$ and morphisms the injective order-preserving maps. A \emph{semisimplicial set} is a functor $X_\bullet \colon \Delta_\mr{inj}^\mr{op} \to \cat{Set}$. Such a functor is equivalent to a collection of sets $X_p$ of \emph{$p$-simplices} for $p \geq 0$ with \emph{face maps} $d_i \colon X_p \to X_{p-1}$ satisfying the face relations. Similarly, a \emph{semisimplicial space} is a functor $X'_\bullet \colon \Delta_\mr{inj}^\mr{op} \to \cat{Top}$. A \emph{semisimplicial map} is a natural transformation of functors $\Delta_\mr{inj}^\mr{op} \to \cat{Set}$ or $\Delta_\mr{inj}^\mr{op} \to \cat{Top}$.\end{definition}

\begin{definition}A \emph{simplicial complex} $Y_\circ$ is a set $Y_0$ of \emph{vertices} and for each $p \geq 1$ a collection $Y_p$ of \emph{$p$-simplices} consisting of $(p+1)$ element subsets of $Y_0$. These must have the property that each $p$-element subset of an element of $Y_p$ is in $Y_{p-1}$. A \emph{simplicial map} is a map of vertex sets that sends simplices to simplices (it may send $p$-simplices to $q$-simplices).\end{definition}

\begin{notation}We use $(-)_\bullet$ for semisimplicial sets or spaces and $(-)_\circ$ for simplicial complexes.\end{notation}

\begin{definition}An \emph{augmentation} for a semisimplicial set or space $X_\bullet$ is a map $\epsilon \colon X_0 \to X_{-1}$ such that $\epsilon d_0 = \epsilon d_1$.\end{definition}

\subsection{Geometric realization} Both semisimplicial sets or spaces and simplicial complexes can be geometrically realized by glueing together simplices using the face maps or face relations, an operation denoted by $||X_\bullet||$ and $|Y_\circ|$ respectively:
\[||X_\bullet|| \coloneqq \bigsqcup_{p \geq 0} X_p \times \Delta^p/\sim_\bullet \qquad \qquad |Y_\circ| \coloneqq \bigsqcup_{p \geq 0} Y_p \times \Delta^p/\sim_\circ\]
Semisimplicial and simplicial maps induce continuous maps on geometric realization. A important result on geometric realization is the realization lemma, Proposition 2.7 of \cite{grwstab1}.

\begin{lemma}\label{lem.geomlevelwise} If $f_\bullet \colon X_\bullet \to Y_\bullet$ is a semisimplicial map between semisimplicial spaces such that $f_p$ is $(n-p)$-connected, then $||f_\bullet||$ is $n$-connected.\end{lemma}

\subsection{Weakly Cohen-Macauley complexes} We will be interested in certain nice simplicial complexes, the so-called \emph{weakly Cohen-Macauley complexes}, a generalization of the notion of a PL manifold.

\begin{definition}Given a $p$-simplex $\sigma$ of $X_\circ$, its \emph{link} $\mr{Link}_\sigma(X)$ is the full subcomplex of all simplices $\tau$ such that $\sigma \cap \tau = \varnothing$ and $\sigma \cup \tau$ is a simplex. Its \emph{star} $\mr{Star}_\sigma(X)$ is the full subcomplex spanned by $\sigma$ and its link.\end{definition}

\begin{definition}\label{def.wcm} A simplicial complex $X_\circ$ is said to be \emph{weakly Cohen-Macauley of dimension $\geq n$} if $|X_\bullet|$ is $(n-1)$-connected and the link of each $p$-simplex is $(n-p-2)$-connected. We denote this by $\mr{wCH}(X_\circ) \geq n$.\end{definition}

\begin{example}\label{exam.inj} An example of a weakly Cohen-Macauley complex is the simplicial complex $\mr{Inj}_\circ(S)$ of injective words on a set $S$. The vertices are the elements of $S$ and a collection $\{s_0,\ldots,s_p\}$ is a $p$-simplex if $s_i \neq s_j$ for $i \neq j$. Note that $|\mr{Inj}_\circ(S)| \cong \Delta^{|S|-1}$. It is thus contractible, as are links of all $p$-simplices except for the $(|S|-1)$-simplex, whose link is empty. From this one deduces that $\mr{wCH}(\mr{Inj}_\circ(S)) \geq |S|-1$.\end{example}

Being weakly Cohen-Macauley is inherited by links; Lemma 2.1 of \cite{grwstab1} says that if $\mr{wCH}(X_\circ) \geq n$, then $\mr{wCH}(\mr{Link}_\sigma(X)) \geq n-|\sigma|-1$.

\subsection{Lifting} One advantage of weakly Cohen-Macauley complexes is the following lifting lemma, a special case of Lemma 2.4 of \cite{perlmutterstab}.

\begin{proposition}\label{prop.lift} Let $f \colon X_\circ \to Y_\circ$ be simplicial map between simplicial complexes, then $\mr{wCM}(X_\circ) \geq n$ if the following conditions are met:
	\begin{enumerate}[(i)]
		\item $f$ has the link lifting property, i.e. if $y \in Y_\circ$ is a vertex and $A$ is a collection of vertices in $X_\circ$ such that $f(a) \in \mr{Link}_y(Y_\circ)$ for all $a \in A$ then there exists a vertex $x \in X_\circ$ such that $f(x) = y$ and $a \in \mr{Link}_x(X_\circ)$ for all $a \in A$,
		\item $\mr{wCM}(Y_\circ) \geq n$,
		\item $f(\mr{Link}_\sigma(X_\circ)) \subset \mr{Link}_{f(\sigma)}(Y_\circ)$ for all simplices $\sigma$ of $X$.
	\end{enumerate}
\end{proposition}

\subsection{Flag complexes}
We will be interested in particular types of simplicial complexes and semisimplicial spaces common to homological stability arguments.

\begin{definition}A \emph{flag complex} is a simplicial complex with the property that $(k+1)$-tuple $\{x_0,\ldots,x_k\}$ of vertices is a $k$-simplex if and only if $\{x_i,x_j\}$ is a 1-simplex for each $i \neq j$.\end{definition}

If two vertices form a 1-simplex they are said to be \emph{compatible}. The complex of injective words is an example of a flag complex. There is also a version of this definition for semisimplicial sets or spaces.

\begin{definition}An \emph{(ordered) topological flag complex} is a semisimplicial space $Y_\bullet$ with the following two properties:
	\begin{enumerate}[(i)]
		\item The map $Y_p \to Y_0 \times \ldots \times Y_0$ is a homeomorphism onto its image, which is an open subset.
		\item The composite $Y_p \to Y_0^p \to (Y_0^p)/\fS_p$ is injective.
		\item An ordered $(p+1)$-tuple $(y_0,\ldots,y_p)$ forms a $p$-simplex if and only if $(y_i,y_j)$ for each $i < j$ lies in $Y_1 \subset Y_0 \times Y_0$.
	\end{enumerate}
\end{definition}

Note that conditions (ii) and (iii) endow the elements of $Y_0$ with a partial order: $y_0 \prec y_1$ if $y_0$ and $y_1$ are compatible and $(y_0,y_1) \in Y_1$. Condition (iii) then says that an ordered $(p+1)$-tuple $(y_0,\ldots,y_p)$ forms a $p$-simplex if and only if $y_i \prec y_{i+1}$ for $0 \leq i \leq p-1$.

To a topological flag complex $Y_\bullet$ one can associate a semisimplicial set $Y^\delta_\bullet$ by forgetting the topology. From this semisimplicial set we can produce a flag complex $Y^\delta_\circ$, which has vertices given by the elements of the set underlying $Y_0$. A $(p+1)$-tuple $\{y_0,\ldots,y_p\}$ of such elements forms $p$-simplex in $Y^\delta_\circ$ if and only if there is a permutation $\sigma \colon \{0,\ldots,p\} \to \{0,\ldots,p\}$ such that $(y_{\sigma(0)},\ldots,y_{\sigma(k)}) \in Y_k$. The uniqueness of the permutation in condition (ii) tells us that $Y^\delta_\bullet$ can be recovered from $Y^\delta_\circ$ and that there is a homeomorphism $||Y_\bullet^\delta|| \cong |Y_\circ^\delta|$.

To prove that a topological flag complex is highly-connected, one may the following consequence of Theorem 2.8 of \cite{perlmutterstab}, which is based on a technique in \cite{grwstab1}.

\begin{proposition}\label{prop.top} Let $X_\bullet$ be a topological flag complex and suppose that $\mr{wCM}(X^\delta_\circ) \geq n$. Then the geometric realization $|X_\bullet|$ is $(n- 1)$-connected.\end{proposition}

\subsection{The subset construction} Let $\mr{sd}(X)_\circ$ denote the barycentric subdivision of $X_\circ$. This is the nerve of the poset of simplices under inclusion, so that $\mr{sd}(X)_\circ$ arise from a simplicial set $\mr{sd}(X)_\bullet$. Let $\mr{sd}_r(X)_\circ$ be the subcomplex spanned by the vertices corresponding to $k$-simplices of $X_\circ$ satisfying $k \geq r$. For example, $\mr{sd}_0(X)_\circ = \mr{sd}(X)_\circ$ and thus its geometric realization is homeomorphic to that of $X_\circ$.

\begin{lemma}\label{lem.subsetconstr} If $\mr{wCH}(X_\circ) \geq n$, then $\mr{sd}_r(X)_\circ$ is $(n-r-1)$-connected.\end{lemma}

\begin{proof}The proof is by induction over $r$, the case $r=0$ being trivial. For the induction step, observe that $|\mr{sd}_{r-1}(X)_\circ|$ is obtained from $|\mr{sd}_r(X)_\circ|$ by attaching the cones on $|\mr{Link}_\sigma(X)|$, for $\sigma$ ranging over the $(r-1)$-simplices of $X_\circ$. These links are $(n-r-1)$-connected, so that the inclusion $|\mr{sd}_r(X)_\circ| \hookrightarrow |\mr{sd}_{r-1}(X)_\circ|$ is $(n-r)$-connected. Since the target is $(n-r)$-connected, this implies the source is $(n-r-1)$-connected.
\end{proof}

\section{The resolution} Quillen's homological stability argument starts with resolving the spaces of interest by a semisimplicial space whose $p$-simplices contain the data to ``undo'' $(p+1)$ applications of the stabilization map. In this section we define this resolution and show it is highly-connected.

\subsection{The resolution} We fix a 3-manifold $M$ with boundary $\partial M$ and chart $\bR \times \bR \hookrightarrow \partial M$. We extend this chart by a collar neighborhood $\bR\times \bR\times [0,\infty) \hookrightarrow M$. 

\begin{definition}Let $\tilde{\cC}_k(M)$ denote the $k$-fold cover of $\cC_k(M)$ with points given by a pair $(\bfc,c)$ of $\bfc \in \cC_k(M)$ and circle $c \in \bfc$.\end{definition}

\begin{definition}The semisimplicial space $X_\bullet(k,M)$ is given as follows:
\begin{itemize}
	\item the space of $0$-simplices consists of the data 
	\[((\bfc,c),\tau,\epsilon,\eta) \in \tilde{\cC}_k(M) \times \bR \times (0,1) \times \mr{Emb}(D^2 \times [0,1],M)\]
	such that (i) $\eta(D^2 \times [0,1]) \cap \bfc = c$, (ii) $\eta^{-1}(c) = \partial D^2 \times \{1\}$, and (iii) near $D^2 \times \{0\}$ the embedding $\eta$ is given by $(x,y,s) \mapsto (\epsilon x+\tau,\epsilon y,s)$ with respect to the collar neighborhood.
	\item a $p$-simplex is a $(p+1)$-tuple $((\bfc_i,c_i),\tau_i,\epsilon_i,\eta_i)$ such that (i') $\eta_i(D^2 \times [0,1]) \cap \eta_j(D^2 \times [0,1]) = \varnothing$ if $i \neq j$, and (ii') $\tau_0< \ldots<\tau_p$.
\end{itemize}
\end{definition}

This is an ordered topological flag complex with augmentation to $\cC_k(M)$. The main result of this section will be the following theorem.

\begin{theorem}\label{thm.resconn} For all path-connected $M$ and $k \geq 0$, the augmentation $||X_\bullet(M)|| \to \cC_k(M)$ is $\lfloor\frac{k-3}{2}\rfloor$-connected.\end{theorem}

\subsection{A complex of disks} To prove Theorem \ref{thm.resconn}, we define spaces of thickened circles and a complex of bounding disks. Let $S^1 \hookrightarrow D^3$ be an unknot, then we can thicken this to an embedding $D^2 \times S^1 \hookrightarrow D^3$. Fix one of these unknotted embeddings of a thickened circle.

\begin{definition}We define \emph{the space of $k$ unlinked thickened circles in $M$} to be
	\[\cC^\mr{th}_k(M) \coloneqq \mr{Emb}^\mr{unl}\left(\bigsqcup_{k}D^2 \times S^1,M\right)/\,\mr{Diff}\left(\bigsqcup_k D^2 \times S^1\right)\]
	where $\mr{Emb}^\mr{unl}(-,-)$ is the subspace of embeddings $\bigsqcup_{k} D^2 \times S^1 \hookrightarrow M$ that extend to an embedding $\bigsqcup_{k} D^3 \hookrightarrow M$ with $D^2 \times S^1 \hookrightarrow D^3$ as above.
\end{definition}

As before, we let $\tilde{\cC}^\mr{th}_k(M)$ be the $k$-fold cover consisting of a pair $(\bfC,C)$ of a $\bfC \in \cC^\mr{th}_k(M)$ and a thickened circle $C \in \bfC$. We let $C(\bfC) \subset M$ denote the union of the thickened circles. The complement of the interior of $C(\bfC)$ is a manifold with boundary. Given a 3-manifold $M$ with boundary $\partial M$, the space of \emph{neat embeddings} of $D^2$ into $M$ is the subspace of the space of injective smooth maps $D^2 \to M$ with the $C^\infty$-topology which are locally modeled on the inclusions $\bR^2 \hookrightarrow \bR^3$ and $\bR \times [0,\infty) \hookrightarrow \bR^2 \times [0,\infty)$.

\begin{definition}Fix a $\bfC \in \cC^\mr{th}_k(M)$. The simplicial complex $D_\circ(\bfC,M)$ is given as follows:
	\begin{itemize}
		\item the vertices consist of pairs 
		\[((\bfC,C),\delta) \in \tilde{\cC}^\mr{th}_k(M) \times  \mr{Emb}^\mr{neat}(D^2,M \setminus \mr{int}(C(\bfC)))\]
		such that (i) $\delta(\partial D^2) \subset \partial C$, and (ii) $\delta|_{\partial D^2} \colon \partial D^2 \to C$ represents a generator of $\pi_1(C) \cong \bZ$,
		\item a $k$-simplex is a $(k+1)$-tuple of $((\bfC,C_i),\delta_i)$ such that $\delta_i(D^2) \cap \delta_j(D^2) = \varnothing$ if $i \neq j$.
	\end{itemize}
\end{definition}

This is a flag complex. Note that in this definition more than one disk may be attached to the boundary of a thickened circle. We will address this issue later.

Two neat embeddings $\delta_0,\delta_1 \colon D^2 \hookrightarrow M \setminus \mr{int}(C(\bfC))$ are \emph{transverse} if they are transverse on $\partial D^2$ and $\mr{int}(D^2)$. Note that like transversality for closed submanifolds, this is an open condition.

\begin{lemma}\label{lem.dbfcmcontractible}We have that $D_\circ(\bfC,M)$ is weakly contractible.\end{lemma}

\begin{proof}By simplicial approximation it suffices to show that for every $i \geq -1$, PL triangulation $L_\circ$ of $S^i$ and simplicial map $f \colon L_\circ \to D_\circ(\bfC,M)$, we can extend $L_\circ$ to a PL triangulation $K_\circ$ of $D^{i+1}$ and $f$ to a simplicial map $K_\circ \to D_\circ(\bfC,M)$. Denote the thickened circle and embedding associated by $f$ to a vertex $v \in L_0$ by $C_v$ and $\delta_v$ respectively. We also write $\delta_v$ for the image of $\delta_v$.
	
We can always find a $((\bfC,C_0),\delta_0)$ with $\delta_0 \colon D^2 \to M  \setminus \mr{int}(C(\bfC))$ transverse to $\delta_v$ for all vertices $v \in L_\circ$: one makes $\delta_0$ transverse to $\delta_v$ using Sard's lemma one $v$ at a time and uses that the condition of transversality is open. We will homotope $f$ into the link of $((\bfC,C_0),\delta_0)$ in $D_\circ(\bfC,M)$ and then cone off the sphere by noting that a link of a vertex is null-homotopic in its vertex, i.e. take the cone on $((\bfC,C_0),\delta_0)$.
	
By transversality $\delta_0 \cap \delta_v$ is a finite disjoint union of circles and arcs. We will first homotope $f$ to remove all circles and then homotope $f$ again to remove all arcs. This is done by a version of a procedure by Hatcher, see Section 2 of \cite{hatcherautfn} or Section 7 of \cite{hanglamthesis}. One starts by picking for $\delta_0$ a \emph{tubular neighborhood}, that is, an embedding $\tilde{\delta}_0 \colon (-1,1) \times D^2 \hookrightarrow M$ with the following properties:
\begin{enumerate}[(i)]
	\item $\tilde{\delta}_0(s,d) = \delta_0(d)$ if $s=0$
	\item $\tilde{\delta}_0((-1,1) \times \partial D^2) \subset \partial C_0$.
\end{enumerate}
Similarly pick tubular neighborhoods $\tilde{\delta}_v$ for $\delta_v$ and impose the following conditions on these tubular neighborhoods:
\begin{enumerate}[(a)]

	\item By shrinking the tubular neighborhoods further, we may assume that if $v$ and $v'$ form a $1$-simplex in $L_\circ$, then $\tilde{\delta}_v$ and $\tilde{\delta}_{v'}$ have disjoint image.
	\item We claim that after shrinking and isotoping the tubular neighborhood, we may assume for each component $P$ of $\delta_v \cap \delta_0$, there exists a $\lambda_P \in (-1,1) \setminus \{0\}$ such that $\tilde{\delta}_v(s,\delta_v^{-1}(p)) = \tilde{\delta}_0(\lambda_P s,\delta_0^{-1}(p))$. This is proven in Lemma \ref{lem.tubulartransversality}.	
	\item Since transversality is an open condition, by shrinking the tubular neighborhoods we may assume that $\tilde{\delta}_v|_{\{s\} \times D^2}$ is transverse to $\tilde{\delta}_0(D^2)$ for all $v$ and $s \in (-1,1)$. 
\end{enumerate}

Now Hatcher's procedure starts.\\

\textbf{Step 1: removing circles.} The proof is by induction over the sum  over all vertices $v \in L_\circ$ of the number of circles in the intersection $\delta_0 \cap \delta_v$. We will show how to reduce this number by $1$.

We start by picking a vertex $v$ and circle $\gamma \subset \delta_0 \cap \delta_v$ which is \emph{innermost}. This means that that there is no circle in $\bigcup_v (\delta_0 \cap \delta_v)$ which is contained in the disk $D_0$ in $\delta_0$ bounded by $\gamma$. Of course there may be circles or arcs in $\bigcup_v (\delta_0 \cap \delta_v)$ intersecting $\gamma$, but note that these come from vertices $v'$ that are \emph{not} in the star of $v$.

Recall that $D_0$ is the disk in $\delta_0$ bounded by $\gamma$, and let $D_v$ be the disk in $\delta_v$ bounded by $\gamma$. Now we intend to replace $\delta_v$ by doing surgery along $D_0$. For an $\epsilon \in (0,1)$  consider the two sets, which are disks by condition (b), given by
\[\tilde{\delta}_v(\{\pm \epsilon\} \times D^2 \setminus \delta_v^{-1}(D_v)) \cup \tilde{\delta}_0(\{\pm \lambda_\gamma \epsilon\} \times \delta_0^{-1}(D_0))\]
one of which must intersect $\delta_0$ in one circle less. For convenience say this is the disk corresponding to $+\epsilon$. It does not intersect $\delta_{v'}$ for any vertex in the star of $v$ (including $v$ itself), as the annulus $\tilde{\delta}_v(\{\epsilon\} \times D^2 \setminus \delta_v^{-1}(D_v))$ must be disjoint $\delta_{v'}$ by (a) and the inner disk $\tilde{\delta}_0(\{\lambda_\gamma \epsilon\} \times \delta_0^{-1}(D_0))$ is disjoint from $\delta_{v'}$ because the circle was innermost. It is also transverse to $\delta_0$ by property (c) (note that it has corners, but the corners do not intersect $\delta_0$ so this makes sense). We can smooth the corners while preserving these properties, as smoothing corners can be done in an arbitrarily small neighborhood of the corners. Finally, we parametrize the result by $\delta_v^\mr{new} \colon D^2 \hookrightarrow M$. 

We claim that $f$ is homotopic to a new map with its value $((\bfC,C_v),\delta_v)$ on $v$ replaced by $((\bfC,C_v),\delta_v^\mr{new})$. To see this, note we can glue the disk $\mr{Link}_v(L_\circ) \ast \Delta^1$ to $L_\circ$ by identifying $\mr{Link}_v(L_\circ) \ast \{0\}$ with $\mr{Link}_v(L_\circ) \ast v$. Then the map $f$ extends to a simplicial map $L_\circ \cup \mr{Link}_v(L_\circ) \ast \Delta^1 \to D_\circ(\bfC,M)$ by sending $\{1\} \in \Delta^1$ to $((\bfC,C_v),\delta_{v}^\mr{new})$, which gives the desired homotopy.

At this point we restart step 1, after picking a new tubular neighborhood for $\delta_{v}^\mr{new}$.\\

\textbf{Step 2: removing arcs.} By repeating step 1 until no circles are left, we may assume that $\bigcup_v (\delta_0 \cap \delta_v)$ consists of only arcs. We will prove that we can remove them by induction over the total number of arcs, by showing how to reduce this number by $1$.

We start by picking a vertex $v$ and an arc $\gamma \subset \delta_0 \cap \delta_v$ which is \emph{outermost}, which means that there is a component of $\delta_0 \setminus \gamma$ which does not contain any arcs. Note that necessary $\delta_0$ and $\delta_v$ are attached to the same circle $C$. Of course there may be arcs in $\bigcup_v (\delta_0 \cap \delta_v)$ which intersect $\gamma$, but as before these come from vertices $v'$ that are not in the star of $v$. We pick a component $K_0$ of $\delta_0 \setminus \gamma$ which does not contain any arcs, in case there is more than one. Its closure is diffeomorphic to a half disk $D^2_+$. 

Now we do a similar surgery construction as before, using a half-disk instead of a disk. Recall that $K_0$ is a half-disk in $\delta_0$ bounded by $\gamma$. There are two half-disks $K_v$ and $K'_v$ in $\delta_v$ bounded by $\gamma$, and we claim that at least one of $(K_0 \cup K_v) \cap C$, $(K_0 \cup K'_v) \cap C$ represents a generator of $\pi_1(C) \cong \bZ$. Without loss of generality this $K_v$ has this property.

We will prove this using a homological argument. We first fix notation $\alpha_v = \delta_v \cap C$, $\beta_v = (K_0 \cup K_v) \cap C$ and $\beta'_v = (K_0 \cup K'_v) \cap C$. These are not oriented, so we pick an orientation for them. Pick a symplectic basis of $H_1(\partial C)$ such that $e_1$ maps to a generator of $H_1(C) \cong \bZ$ and $e_2$ maps to $0 \in H_1(C)$. Without loss of generality $[\alpha_v] = e_1+a_v e_2$, and we also write $[\beta_v] = b_1 e_1 + b_2 e_2$ and $[\beta'_v] = b'_1 e_1 + b_2' e_2$. By construction $\beta_v$ and $\beta_v'$ are homotopic to loops that do not intersect $\alpha_v$, so that $[\alpha_v] \cdot [\beta_v] = b_2-a_v b_1 =  0 =  b_2'-a_v b_1' = [\alpha_v] \cdot [\beta_v']$. By construction $[\beta_v]-[\beta_v'] = [\alpha_v]$, so that $b_1-b_1' = 1$ and $b_2 - b_2' = a_v$. The only solutions these equations are $[\beta_v] = (c+1)(e_1+a_v e_2)$ and $[\beta_v']  = c(e_1+a_v e_2)$. Finally, we remark that both $\beta_v$ and $\beta_v'$ are embedded. This implies that either both their coefficients are $0$, or the greatest common divisor of their coefficients is $1$, see Theorem 2.C.2 of \cite{rolfsen}. We conclude that $c = 0,-1$, proving the claim.

 Now we do surgery along the half-disk. For $\epsilon > 0$ small consider the two disks
\[\tilde{\delta}_v(\{\pm \epsilon\} \times D^2 \setminus \delta_v^{-1}(K_v)) \cup \tilde{\delta}_0(\{\pm \lambda_\gamma \epsilon\} \times K_0)\]
one of which must intersect $\delta_0$ in one arc less.  For convenience say this is the disk corresponding to $+\epsilon$. It does not intersect $\delta_{v'}$ for any vertex in the star of $v$ (including $v$ itself), as the half-disk $\tilde{\delta}_v(\{\epsilon\} \times D^2 \setminus \delta_v^{-1}(K_v))$ must be disjoint from $\delta_{v'}$ by (b) and the half-disk $\tilde{\delta}_0(\{\lambda_\gamma \epsilon\} \times \delta_0^{-1}(K_0))$ is disjoint from $\delta_{v'}$ because the arc was outermost. Now the argument finishes as in Step 1.
\end{proof}

\begin{lemma}\label{lem.tubulartransversality} Suppose we have two codimension one embeddings $\delta_0 \colon N_0 \hookrightarrow M$ and $\delta_1 \colon N_1 \hookrightarrow M$ of compact submanifolds, which intersect transversally. Suppose further we are given tubular neighborhoods $\tilde{\delta}_i \colon (-1,1) \times N_i \hookrightarrow M$ extending the embeddings $\delta_i$. Then we may isotope $\tilde{\delta}_1$ through $\tilde{\delta}_{1,t}$ for $t \in [0,1]$ such that (i) $\tilde{\delta}_{1,t} \subset \tilde{\delta}_1$ for all $t$, (ii) $\tilde{\delta}_{1,t}|_{N_1 \times \{0\}} = \delta_1$ for all $t$, and (iii) for each component $P$ of $N_0 \cap N_1$ there exists an $\lambda_P \in (-1,1) \setminus \{0\}$ such that $\tilde{\delta}_{1,1}(s,\delta_{1}^{-1}(p)) = \tilde{\delta}_0(\lambda_P s,\delta_0^{-1}(p))$.
\end{lemma}

\begin{proof}We arrange this one component $P$ at a time,working in a small enough neighborhood of $P$ to not modify $\tilde{\delta}_1$ near other components $P'$ of $N_0 \cap N_1$. By shrinking $\tilde{\delta}_1$ we may assume that $\tilde{\delta}_1((-1,1) \times P) \subset \tilde{\delta}_0((-1,1) \times N_0)$. This gives us a map $\pi_2 \circ \tilde{\delta}_0^{-1} \circ \tilde{\delta}_1 \colon (-1,1) \times P \to (-1,1)$ which equals $0$ when restricted to $\{0\} \times P$. Consider the adjoint map $\gamma_P \colon P \to \mr{Map}((-1,1),(-1,1))$. By transversality $\gamma_P(p)$ has non-zero derivative at $0$ for all $p \in P$ and thus there exists a family $\ell_{P,t} \colon P \to \mr{Emb}((-1,1),(-1,1))$ for $t \in [0,1]$ starting at the identity such that (i) $\ell_{P,t}(p)$ fixes $0$ for all $t$ and $p$, and (ii) $\gamma_P \circ \ell_{P,1}$ is constant equal to $(s \mapsto \lambda_P s)$ for some $\lambda_P \in (-1,1) \setminus \{0\}$. We may extend $\ell_{P,t}$ to a compactly supported family $\hat{\ell}_t \colon N_1 \to \mr{Emb}((-1,1),(-1,1))$ that is supported in neighborhood of $P$ disjoint from all other components of $N_0 \cap N_1$. 

Now consider the family of maps $\tilde{\delta}_{1,t} = \tilde{\delta}_1 \circ (\hat{\ell}_t \times \mr{id}) \colon (-1,1) \times N_1 \to M$, which is given by reparametrizing $\tilde{\delta}_1$ in the normal direction. By transversality it is an embedding near $\{0\} \times N_1$ for all $t \in [0,1]$, and its restriction to $\{0\} \times N_1$ equals $\delta_1$ for all $t$. 

By first shrinking the tubular neighborhood and then applying the homotopy $\tilde{\delta}_{1,t}$, we may isotope $\tilde{\delta_1}$ so that $\tilde{\delta}_0^{-1} \circ \tilde{\delta}_{1,1}$ is of the form $(s,p) \mapsto (\lambda_P s,G'_P(s)(p))$ with $G'_P \colon (-1,1) \to \mr{Emb}(P,N_0)$ satisfying $G'_P(0)=\mr{id}$. Using isotopy extension and the fact that the space of embeddings of $P$ into $N_0$ is locally contractible (it is an infinite-dimensional manifold, see e.g. Chapter 13 of \cite{michor}), we can find an $\epsilon > 0$ and map $\phi \colon (-\epsilon,\epsilon) \to \mr{Diff}(N_0)$ sending $0$ to the identity and supported near $P$, so that $(s,p) \mapsto (\lambda_P s,\phi(s)G'_P(s)(p))$ is given by $(s,p) \mapsto (\lambda_Ps,p)$. After shrinking the tubular neighborhood, we get an isotopy $\tilde{\delta}_0 \circ (\mr{id},\phi(t \cdot -)) \circ \tilde{\delta}_0^{-1} \circ \tilde{\delta}_{1,t}$ which starts at $\tilde{\delta}_1$ and ends at a tubular neighborhood with the desired property.
\end{proof}

We will now show that without losing too much connectivity we can assume at most one disk attaches to each thickened circle.

\begin{definition}Let $Y_\circ(\bfC,M) \subset D_\circ(\bfC,M)$ denote the subcomplex where we additionally require that in a simplex $\{(\bfC,C_i),\delta_i)\}$ we have that $C_i \neq C_j$ if $i \neq j$.\end{definition}

\begin{lemma}\label{lem.ylcm} We have that $\mr{wCM}(Y_\circ(\bfC,M)) \geq k-1$.
\end{lemma}

\begin{proof}Consider the link of a simplex $\sigma$ in $Y_\circ(\bfC,M)$. If $C(\sigma)$ denotes the union of the thick circles in $\sigma$ and $\delta(\sigma)$ the union of the disks in $\sigma$, then this is isomorphic to 
	\[Y_\circ(\bfC \setminus C(\sigma),M \setminus (C(\sigma) \cup \delta(\sigma)))\] 
where $\bfC \setminus C(\sigma)$ consists of $k-|\sigma|-1$ thick circles. These are still unknotted, as by a surgery argument like in Lemma \ref{lem.dbfcmcontractible} we can make a bounding disk disjoint from $\delta(\sigma)$. Hence to prove that $\mr{wCM}(Y_\circ(\bfC,M)) \geq k-1$, it suffices to prove that $Y_\circ(\bfC,M)$ is $(k-2)$-connected.
	
This is a badness argument, by induction over $k$. The initial case $k=1$ is trivial. For the induction step, by simplicial approximation we can assume there is a PL triangulation $L_\circ$ of $S^i$ with $i \leq k-2$ and a simplicial map $f \colon L_\circ \to Y_\circ(\bfC,M)$. We need to prove this extends over a disk. Since $D_\circ(\bfC,M)$ is contractible, by simplicial approximation there is a PL triangulation $K_\circ$ of $D^{i+1}$ extending $L_\circ$ and simplicial map $F \colon K_\circ \to D_\circ(\bfC,M)$ extending $f$:
\[\xymatrix{S^i \cong L_\circ \ar[r]^-f \ar[d] & Y_\circ(\bfC,M) \ar[d] \\
D^{i+1} \cong K_\circ \ar[r]_-F & D_\circ(\bfC,M)}\]

We call a simplex $\sigma$ of $K_\circ$ \emph{bad} if all of the thick circles $C_i$ in $F(\sigma)$ are equal, so that $Y_\circ(\bfC,M) \subset D_\circ(\bfC,M)$ is the subcomplex without bad simplices. We show how to decrease the number of bad simplices in $K_\circ$ of maximal dimension by 1, changing $K_\circ$ and $F$ in the process. Eventually this process eliminates all bad simplices.
	
Suppose that $\sigma$ is a bad simplex of maximal dimension $1 \leq j \leq i$. Since $\sigma$ is maximal, $\mr{Link}_{\sigma}(K_\circ)$ maps into the subcomplex $Y_\circ(\bfC \setminus C(F(\sigma)),M \setminus (C(F(\sigma)) \cup \delta(F(\sigma))))$. Since $K_\circ$ is a PL triangulation of a PL manifold and $\sigma$ necessarily lies in its interior, we have $\mr{Link}_{\sigma}(K_\circ) \cong S^{i-j-1}$ and we can estimate $i-j-1 \leq k-j-3$.

 Note that the manifold $M \setminus (C(F(\sigma)) \cup \delta(F(\sigma)))$ is a disjoint union of $r$ manifolds $M_\ell$, each containing a configuration $\bfC_\ell \coloneqq \bfC \cap M_l$ of $k_\ell$ thickened circles, and that $\sum k_\ell = k-1$. A surgery argument as in the proof of Lemma \ref{lem.dbfcmcontractible} shows that $\bfC_\ell$ is unlinked in $M_\ell$. Hence we have that
 \[Y_\circ(\bfC \setminus C(F(\sigma)),M \setminus (C(F(\sigma)) \cup \delta(F(\sigma)))) \cong Y_\circ(\bfC_1,M_1) \ast \cdots \ast Y_\circ(\bfC_r,M_r)\]
By induction this has a connectivity of $-2+\sum_{\ell=1}^r k_\ell = -2+\sum k_\ell \leq k-3$ and thus we can find a triangulated PL disk $J_\circ$ of dimension $i-j$ with $\partial J_\circ = \mr{Link}_{\sigma}(K_\circ)$ and a map $\bar{F} \colon J_\circ \to Y_\circ(\bfC \setminus C(F(\sigma)),M \setminus (C(F(\sigma)) \cup \delta(F(\sigma))))$ which extends $f|_{\partial J_\circ}$. Consider the new map
	\[F' \colon (K_\circ \backslash \mr{int}(\mr{Star}_{K_\circ}(\sigma))) \cup_{\partial \mr{Star}_{K_\circ}(\sigma)} (\partial \sigma \ast J_\circ) \to D_\circ(\bfC,M)\]

Its domain is a PL disk, its restriction to the boundary equal $f$, and we removed the bad simplex $\sigma$. We did not introduce any new bad simplices of dimension $\geq j$, as all new simplices are of the form $\tau \ast \omega$ with $\tau \subset \partial \sigma$ and $\omega \subset J_\circ$, and that all thickened circles in $\bar{F}(\omega)$ are distinct.
\end{proof}

\subsection{Finding tethers} We want a version of $Y_\circ(\bfC,M)$ where circles are tethered to the boundary. We restrict to path-connected $M$, as before with a chart $\bR \times \bR \hookrightarrow \partial M$ and collar $\bR \times \bR \times [0,\infty) \hookrightarrow M$. We also pick a Riemannian metric on $M$ that equals the standard Euclidean metric on this collar. Let $\mr{Emb}^\mr{fr}([0,1],M)$ denote the space of embedded arcs with trivialization of their $2$-dimensional normal bundle, a subspace of $\mr{Map}([0,1] \times \bR^2,TM)$.

\begin{definition}Let $YT_\circ(\bfC,M)$ be the simplicial complex given as follows:
	\begin{itemize}
		\item vertices given by a triple 
		\[(v,\tau,\gamma) \in Y_0(\bfC,M) \times \bR \times \mr{Emb}^\mr{fr}([0,1],M \setminus C(\bfC))\]
		of a vertex $v = ((\bfC,C),\delta)$ of $Y_\circ(\bfC,M)$,  a real number $\tau \in \bR$ and an embedded arc $\gamma$ such that (i) $\gamma([0,1))$ is disjoint from the disk $\delta(D^2)$, (ii) $\gamma(1) = \delta(0)$, $T\gamma(1)$ is orthogonal to $T\delta(0)$ and the trivialization coincides with that coming from $T\delta(D^2)$, (iii) near $0$ the embedding $\gamma$ is given by $s \mapsto (\tau,0,s)$ and the trivialization coincides with the standard one.
		\item a $p$-simplex is a $(p+1)$-tuple $(v_i,\tau_i,\gamma_i)$ such that (i') the $v_i$ form a $p$-simplex in $Y_\circ(\bfC,M)$, (ii') $(\delta_i(D^2) \cup \gamma_i([0,1]))\cap(\delta_j(D^2) \cup \gamma_j([0,1])) = \varnothing$ if $i\neq j$, and (iii') $\tau_0< \ldots<\tau_p$.
			\end{itemize}
	\end{definition}

Note this is an ordered flag complex, and thus there is an associated semisimplicial set $YT_\bullet(\bfC,M)$ with homeomorphic geometric realization. To investigate the connectivity of $YT_\bullet(\bfC,M)$ we introduce an intermediate bisemisimplicial set. Note that $\sigma \in \mr{sd}_r(Y)_\bullet$ is a chain of simplices $\tau_0 \subset \ldots \subset \tau_p$ of $Y_\circ$ and we let $\min(\sigma) \coloneqq \tau_0$ denote the smallest one.

\begin{definition}The bisemisimplicial set $YT_{\bullet,\bullet}(r,\bfC,M)$ has $(p,q)$-simplices given by a pair of a $p$-simplex $\sigma$ of $\mr{sd}_r(Y)_\bullet(\bfC,M)$ and $q$-simplex $\tau$ of $YT_\bullet(\bfC,M)$ such that $C(\tau) \subset C(\min(\sigma))$.\end{definition}

Note that there are augmentations
\begin{equation}\label{eqn.zigzag} \mr{sd}_r(Y)_\bullet(\bfC,M) \overset{\epsilon_1}{\longleftarrow} YT_{\bullet,\bullet}(r,\bfC,M) \overset{\epsilon_2}{\longrightarrow} YT_\bullet(\bfC,M)\end{equation}
each given by forgetting the other simplicial direction.

\begin{lemma}The realization of $\epsilon_1$ is $(r-1)$-connected.\end{lemma}

\begin{proof}Given a simplex $\sigma$ of $\mr{sd}_r(Y)_\bullet(\bfC,M)$ the fiber of $\epsilon_1$ over it is the semisimplicial subset of $YT_\bullet(\bfC,M)$ whose disks coincide with elements of $\delta(\min(\sigma))$.  There is a corresponding simplicial complex with homeomorphic geometric realization which admits a simplicial map to the simplicial complex $\mr{Inj}_{\circ}(\delta(\max(\sigma)))$. Example \ref{exam.inj} says that $\mr{wCH}(\mr{Inj}_{\circ}(\delta(\min(\sigma)))) \geq r$. 
	
We will apply Proposition \ref{prop.lift} to show that the fiber is $(r-1)$-connected. Only condition (i) is not obvious. In words, we need to show that given a disk $\delta_0$ in $\delta(\min(\sigma))$ and a collection of vertices whose disks are in $\delta(\min(\sigma))$, but not equal to $\delta_0$, we can find a framed arc connecting $\delta_0$ to the boundary. For this it suffices to show that $\delta_0$ is not disconnected from the boundary chart by the disks or arcs of the other vertices. The disks of the other vertices can not disconnect $M$, since they form a simplex in $Y_\circ(\bfC,M)$ and thus up to homotopy removing them from $M$ is the same as removing points. The arcs of the other vertices can not disconnect $M$ since they are 1-dimensional subsets.

By Lemma \ref{lem.geomlevelwise}, the geometric realization of the map is $(r-1)$-connected.\end{proof}

\begin{lemma}The realization of $\epsilon_2$ is $(k-r-2)$-connected.\end{lemma}

\begin{proof}The fiber of $\epsilon_2$ over a simplex $\tau$ in $YT_\bullet(\bfC,M)$ is given by the semisimplicial subset of $\mr{sd}_r(Y)_\bullet(\bfC,M)$ spanned by the vertices $v$ such that the disks $\delta(\tau)$ of $\tau$ are among those in $v$. There are two cases. The first is $r\leq |\tau|$, and then there is an initial object given by $\delta(\tau)$, so that the fiber is contractible. The second is $r>|\tau|$, and then simplices $v$ of $Y_\circ(\bfC,M)$ such that the disks $\delta(\tau)$ are among those in $v$ are determined by their complement. Thus the complex is isomorphic to \[\mr{sd}_{r-|\tau|-1}(\mr{Link}_\tau(Y))_\bullet(\bfC \setminus C(\tau),M \setminus (\delta(\tau) \cup \gamma(\tau))))\]
which is $(k-|\tau|-3-r+|\tau|+1-1) = (k-r-3)$-connected by Lemma \ref{lem.subsetconstr}. By Lemma \ref{lem.geomlevelwise}, the geometric realization of the map is $(k-r-2)$-connected.\end{proof}

\begin{lemma}We have that $YT_\bullet(\bfC,M)$ is $\min(k-r-3,r-2)$-connected.\end{lemma}

\begin{proof}In the zigzag (\ref{eqn.zigzag}) the left-hand side is $(k-r-3)$-connected by Lemma's \ref{lem.subsetconstr} and \ref{lem.ylcm}, and the left map is $(r-1)$-connected, so the middle term is $\min(k-r-3,r-2)$-connected. Similarly, as the right map is $(k-r-2)$-connected the right-hand side is $\min(k-r-3,r-2)$-connected.
\end{proof}

Now we may set $r = \lfloor \frac{k-1}{2} \rfloor$ to see that $YT_\bullet(x,M)$ is $\lfloor\frac{k-5}{2}\rfloor$-connected. Note that the links of simplices of $YT_\bullet(\bfC,M)$ are essentially of the same form, differing only in where arcs can attach to the boundary. A similar argument thus gives that the link of a $p$-simplex is $\lfloor\frac{k-p-6}{2}\rfloor$-connected. Since $\lfloor\frac{k-p-6}{2}\rfloor \geq \lfloor\frac{k-5}{2}\rfloor-p-1$, we conclude that $\mr{wCH}(YT_\bullet(x,M)) \geq \lfloor\frac{k-3}{2}\rfloor$:

\begin{proposition}\label{prop.ytconn} For all path-connected 3-manifolds $M$ and $\bfC \in \cC^\mr{th}_k(M)$, we have that $\mr{wCH}(YT_\bullet(\bfC,M)) \geq \lfloor\frac{k-3}{2}\rfloor$.\end{proposition}

\subsection{The proof of Theorem \ref{thm.resconn}} We finish with the proof of Theorem \ref{thm.resconn}, which says that the augmentation map
\[||X_\bullet(k,M)||\lra \cC_k(M)\]
is $\lfloor \frac{k-3}{2}\rfloor$-connected. The isotopy extension theorem implies that for each $p$ the map $X_p(k,M) \to \cC_k(M)$ is a Serre fibration. Then Lemma 2.1 of \cite{oscarresolutions} says that the same is the case for the map $||X_\bullet(k,M)||\to \cC_k(M)$. Thus it suffices to prove that the fibers $||X_\bullet(\bfc,k,M)||$ over $\bfc \in \cC_k(M)$ are $\lfloor \frac{k-5}{2}\rfloor$-connected. This is done in the remainder of this subsection by working towards $X_\bullet(\bfc,k,M)$ through a number of intermediate semisimplicial spaces.

\begin{definition}Let $M$ be path-connected, $k \geq 0$ and $\bfC \in \cC^\mr{th}_k(M)$. The semisimplicial space $XT^\mr{th}_\bullet(\bfC,k,M)$ is given as follows:
	\begin{itemize}
		\item the space of $0$-simplices consists of data 
		\[((\bfC,C),\tau,\gamma,\delta) \in \tilde{\cC}^\mr{th}_k(M) \times \bR \times \mr{Emb}^\mr{fr}([0,1],M \setminus C(\bfC)) \times \mr{Emb}^\mr{neat}(D^2,M\setminus \mr{int}(C(\bfC)))\]
		such that (i) $\delta(\partial D^2)$ lies in $\partial C$ and $\delta|_{\partial D^2}$ represents a generator of $\pi_1(C)$, (ii) $\gamma([0,1))$ is disjoint from the disk $\delta(D^2)$, (iii) $\gamma(1) = \delta(0)$, $T\gamma(1)$ is orthogonal to $T\delta(0)$ and the trivialization coincides with that coming from $T\delta(D^2)$, and (iv) near $0$ the embedding $\gamma$ is given by $s \mapsto (\tau,0,s)$ and the trivialization coincides with the standard one.
		\item a $p$-simplex is a $(p+1)$-tuple of $((\bfC_i,C_i),\tau_i,\gamma_i,\delta_i)$ such that (i') $C_i \neq C_j$ if $i \neq j$, (ii') $(\delta_i(D^2) \cup \gamma_i([0,1]))\cap(\delta_j(D^2) \cup \gamma_j([0,1])) = \varnothing$ if $i \neq j$, and (iii') $\tau_0< \ldots<\tau_p$.
	\end{itemize}
\end{definition}

\begin{lemma}For all path-connected $M$, $k \geq 0$ and $\bfC \in \cC^\mr{th}_k(M)$, the realization of the semisimplicial space $XT^\mr{th}_\bullet(\bfC,k,M)$ is $\lfloor\frac{k-5}{2}\rfloor$-connected.\end{lemma}

\begin{proof}Apply Proposition \ref{prop.top} using Proposition \ref{prop.ytconn}.
\end{proof}

\begin{definition}Let $M$ be path-connected, $k \geq 0$ and $\bfC \in \cC^\mr{th}_k(M)$. The semisimplicial space $X_\bullet^\mr{th}(\bfC,k,M)$ is given as follows:
	\begin{itemize}
		\item the space of $0$-simplices consists of the data 
		\[((\bfC,C),\tau,\epsilon,\eta) \in \tilde{\cC}^\mr{th}_k(M) \times \bR \times (0,1) \times \mr{Emb}(D^2 \times [0,1],M \setminus \mr{int}(C(\bfC)))\]
		such that (i) $\eta|_{D^2 \times \{1\}}$ is neat, (ii) $\eta(D^2 \times [0,1]) \cap C(\bfC) \subset C$, (iii) $\eta^{-1}(C) = \partial D^2 \times \{1\}$, (iii) $\eta|_{\partial D^2 \times \{1\}}$ represents a generator of $\pi_1(C)$, and (iv) near $D^2 \times \{0\}$ the embedding $\eta$ is given by $(x,y,s) \mapsto (\epsilon x+\tau,\epsilon y,s)$.
		\item a $p$-simplex is a $(p+1)$-tuple of $((\bfC_i,C_i),\tau_i,\epsilon_i,\eta_i)$ such that (i') $C_i \neq C_j$ if $i \neq j$, (ii')  $\eta_i(D^2 \times [0,1]) \cap \eta_j(D^2 \times [0,1]) = \varnothing$ if $i \neq j$, and (iii') $\tau_0< \ldots<\tau_p$.
	\end{itemize}
\end{definition}

Note that there is a semisimplicial map
\begin{align*}X^\mr{th}_\bullet(\bfC,k,M) &\to XT_\bullet(\bfC,k,M) \\
((\bfC,C_i),\tau_i,\epsilon_i,\eta_i) &\mapsto ((\bfC,C_i),\tau_i,\eta_i|_{\{0\} \times [0,1]},\eta_i|_{D^2 \times \{1\}})
\end{align*}
where the normal bundle to $\eta_i|_{\{0\} \times [0,1]}$ is trivialized by the derivative of $\eta_i$ at $\{0\} \times [0,1]$ in $D^2 \times [0,1]$.

\begin{lemma}\label{lem.xthtoxtweq} The map $||X^\mr{th}_\bullet(\bfC,k,M)|| \to ||XT_\bullet(\bfC,k,M)||$ is a weak equivalence.
\end{lemma}

\begin{proof}By Lemma \ref{lem.geomlevelwise} it suffices to prove that the map is a levelwise weak equivalence. By isotopy extension the map restricting the $p$-simplices of both sides to $(p+1)$-tuple of disks is a fibration, so we may assume the disks to be fixed. Similarly, we may assume the arcs to be fixed. Thus it suffices show that the space of embeddings $D^2 \times [0,1] \hookrightarrow M \setminus \mr{int}(C(\bfC))$ fixed on $D^2 \times \{0,1\} \times \{0\} \times [0,1]$ is weakly equivalent to the space of trivializations of the normal bundle to $[0,1] \to  M \setminus \mr{int}(C(\bfC))$ standard at the endpoints. The space of embeddings $D^2 \times [0,1] \hookrightarrow M \setminus \mr{int}(C(\bfC))$ equal to a fixed embedding on $(D^2 \times \{0,1\}) \cup (\{0\} \times [0,1])$ is weakly equivalent to the space of such immersions, and Smale-Hirsch says this is weakly equivalent to the space of trivializations.\end{proof}

Next, we first pick a diffeomorphism from each thick circle $C_i \in \bfC$ to $D^2 \times S^1$. This in particular defines a smooth map $\pi_i \colon \partial C_i \to S^1$ by projecting $D^2$ to $0$.

\begin{definition}We let $X^\mr{th,st}_\bullet(\bfC,k,M) \subset X^\mr{th}_\bullet(\bfC,k,M)$ be the subsemisimplicial space consisting of $((\bfC,C_i),\tau_i,\gamma_i,\delta_i)$ such that the composite map $\pi_i \circ \delta|_{\partial D^2}$ is a submersion.\end{definition}

\begin{lemma}The map $||X^\mr{th,st}_\bullet(\bfC,k,M)|| \to ||X^\mr{th}_\bullet(\bfC,k,M)||$ is a weak equivalence.\end{lemma}

\begin{proof}By Lemma \ref{lem.geomlevelwise} it suffices to prove it is a levelwise weak equivalence. By the isotopy extension theorem it suffices to prove that if $E_i \subset \mr{Emb}(S^1,\partial C_i)$ denotes the subspace of embeddings representing a generator of $\pi_1(C)$, then the inclusion of the further subspace $E^\mr{st}_i \subset E_i$ of $\phi$ such that $\pi_i \circ \phi_i$ is a submersion, is a weak equivalence. But both have the same set of components, and in both cases the components is weakly equivalent to the subspace of embeddings that are geodesics for the flat metric, e.g. by a curve shortening argument \cite{grayson}.\end{proof}

Next we note that if we pick disjoint embeddings $\bR^2 \times S^1 \hookrightarrow M$ extending diffeomorphism of $C_i$ with $D^2 \times S^1$, there is a collapse map $r \colon M \to M$ obtained as follows. Let $r_0 \colon [0,\infty) \to [0,\infty)$ be a smooth map that is the identity on $(2,\infty)$, an embedding on $(1,2]$ and constant $0$ on $[0,1]$. In the local model consider the map given by $\bR^2 \times S^1  \ni (r,\theta,\phi) \mapsto (r_0(r),\theta,\phi) \in \bR^2 \times S^1$. To obtain $r$ we extend $k$ of these maps by the identity. 

Let $\bfC \in \cC^\mr{tk}_k(M)$ be a thickening of $\bfc \in \cC_k(M)$, which always exists by the tubular neighborhood theorem. Then there is a semisimplicial map
\begin{align*}X^\mr{th}_\bullet(\bfC,k,M) &\to X_\bullet(\bfc,k,M) \\
((\bfC,C_i),\tau_i,\epsilon_i,\eta_i) &\mapsto ((\bfc,r(C_i)),\tau_i,\epsilon_i,r \circ \eta_i)
\end{align*}

\begin{lemma}The map $||X^\mr{th}_\bullet(\bfC,k,M)|| \to ||X_\bullet(\bfc,k,M)||$ is a weak equivalence.\end{lemma}

\begin{proof}In fact, it is a levelwise homeomorphism by blowing up along the circles and identifying the new boundary, given by unit normal bundle of the circles, with $\partial C$.\end{proof}

\section{The proof of the main theorem} We will now give the standard spectral sequence argument that completes the proof of Theorem \ref{thm.main}, as used in e.g. \cite{hatcherwahl, palmerthesis, grwstab1, RW,perlmutterstab}. Thus uses the \emph{geometric realization spectral sequence}. For an augmented semisimplicial space $X_\bullet$, this is given by
\[(E^1_{p,q},d^1) = \left(H_q(X_p),{\textstyle \sum}_i (-1)^i (d_i)_*\right) \Rightarrow H_{p+q+1}(X_{-1},||X_\bullet||)\]
That is, the $E^1$-page has non-zero columns for $p \geq -1$, given by the homology of the $p$-simplices. The $d^1$-differential $d^1 \colon E^1_{p+1,q} \to E^1_{p,q}$ is  the alternating sum of the maps induced by the face maps when $p \geq 0$, and is the map induced by augmentation when $p = -1$.

We will apply this to the augmented semisimplicial space $X_\bullet(k,M)$ and in the remainder of this section we will do the following:
\begin{enumerate}
	\item Identify the target of the spectral sequence.
	\item Identify the $E^1$-page.
	\item Identify the $d^1$-differential.
	\item Finish the proof of the homological stability argument.
\end{enumerate}

\subsection{Identifying the target} This first step is easy given the results we have already proven. The following is a direct consequence of Theorem \ref{thm.resconn}.

\begin{proposition}\label{prop.sstarget} We have that 
	\[H_{*+1}(\cC_k(M),||X_\bullet(k,M)||) = 0\]
 for $* \leq \lfloor \frac{k-5}{2}\rfloor$.\end{proposition}

\subsection{Identifying the $E^1$-page} The next step is to identify the $E^1$-page. The resolution was chosen to make this particularly easy.

\begin{proposition}\label{prop.sse1page} For all $p \geq 0$ we have that
	\[X_p(k,M) \simeq \cC_{k-p-1}(M)\]
\end{proposition}

\begin{proof}Consider the subspace $X_p(k,M)$ where $\eta_i$ is equal to one of the standard maps 
	\[\eta_{\epsilon_i,\tau_i} \colon (x,y,s) \mapsto (\epsilon_i \cdot x+\tau_i,\epsilon_i \cdot y,s)\]
The inclusion into $X_p(k,M)$ is a weak equivalence: pull the circles back along $\eta_i$ to the boundary, reparametrizing $\eta_i$ in the process, and extend $\eta_i$ back out. A linear interpolation gives a deformation retract of this subspace onto those $((\bfc,c_i),\tau_i,\epsilon_i,\eta_{\epsilon_i,\tau_i})$ with $\tau_i = i$ and $\epsilon_i = 1/4$. This is homeomorphic to $\cC_{k-p-1}(M_p)$ where $M_p$ is obtained from $M$ by removing the interior of the cylinders $\eta_{1/4,i}(D^2 \times [0,1])$ for $0 \leq i \leq p$. The interior of the manifold $M_p$ is diffeomorphic to interior of $M$, so that $\cC_{k-p-1}(M_p)$ is homeomorphic to $\cC_{k-p-1}(M)$.
\end{proof}

\subsection{Identifying the $d^1$-differential} The proof of Proposition \ref{prop.sse1page} makes it easy to identify the $d^1$-differential.

\begin{proposition}\label{prop.ssd1diff} We have that
	\[d^1 \colon E^1_{p,q} \cong H_*(\cC_{k-p-1}(M)) \lra E^1_{p-1,q} \cong H_*(\cC_{k-p}(M))\]
is given by $t_*$ if $p$ is even and $0$ is $p$ is odd.
\end{proposition}

\begin{proof}It suffices to note that on the subspace $\cC_{k-p-1}(M_p) \hookrightarrow X_p(k,M)$ the face map $d_i$ is given by adding the $i$th cylinder $\eta_{1/4,i}(D^2 \times I)$ back to $M_p$ and adding the circle $\eta_{1/4,i}(\partial D^2 \times \{1\})$ to the configuration of circles. Each of these maps is easily seen to be homotopic to $t$.\end{proof}

\subsection{Finishing the proof} We now finish the proof of the homological stability argument. We give a small sharpening afterwards.

\begin{proposition}We have that $H_*(\cC_{k}(M),\cC_{k-1}(M)) = 0$ for $* \leq \lfloor \frac{k-4}{2}\rfloor$.\end{proposition}

\begin{proof}[Proof of Theorem \ref{thm.main}] We want to prove that 
	\[t_* \colon H_*\left(\cC_{k-1}(M) \right) \lra H_*\left(\cC_{k}(M)\right)\]
	is an isomorphism in the range $* \leq \lfloor\frac{k-6}{2}\rfloor$ and a surjection in the range $* \leq \lfloor\frac{k-4}{2}\rfloor$, using the geometric realization spectral sequence associated to the resolution $X_\bullet(k,M)$. We only give an outline of the argument, since it is well-known, see e.g. \cite{hatcherwahl, palmerthesis, grwstab1, RW,perlmutterstab}
	
	By Proposition \ref{prop.sse1page}, the augmented geometric realization spectral sequence for $X_\bullet(k,M)$ has $E^1$-page given by
	\[E^1_{p,q} = \begin{cases} H_q(\cC_k(M)) & \text{if $p = -1$} \\
	H_q(\cC_{k-p-1}(M))) & \text{if $p\geq 0$}\end{cases}\]
	By Proposition \ref{prop.sstarget}, it converges to $0$ in the range $p+q \leq  \lfloor \frac{k-5}{2} \rfloor$. Using Proposition \ref{prop.ssd1diff}, the $d^1$-differential as alternatively $t_*$ and $0$. Thus the $E^2$-page has columns given by
	\[E^2_{p,q} = \begin{cases} \mr{coker}[t_* \colon H_q(\cC_{k-p-2}(M)) \to H_q(\cC_{k-p-1}(M))] & \text{if $p$ is odd} \\
	\mr{ker}[t_* \colon H_q(\cC_{k-p-1}(M)) \to H_q(\cC_{k-p}(M))] & \text{if $p$ is even}\end{cases}\] 
	
	The proof is then by induction over $k$ of the statement that $t_ *$ is an isomorphism for $* \leq \lfloor\frac{k-6}{2}\rfloor$ and a surjection for $* \leq \lfloor\frac{k-4}{2}\rfloor$. By the induction hypothesis the $E^2$-page vanishes when $p \geq 1$ and $p+q \leq \lfloor \frac{k-4}{2} \rfloor$. This and the fact that the spectral sequence converge to $0$ in the range  $p+q \leq  \lfloor \frac{k-5}{2}\rfloor$, are the input to a spectral sequence argument that concludes in the vanishing of the columns $p=-1,0$ in a range, proving the inductive step.
\end{proof}

Finally, we show that $t_*$ is always injective.

\begin{lemma}The stabilization map $t_* \colon H_*(\cC_{k}(M)) \to H_*(\cC_{k+1}(M))$ is injective for all $k \geq 0$.\end{lemma}

\begin{proof}The argument is similar to that for configuration spaces in Section 7 of \cite{RW}. It follows directly by applying Lemma 2 of \cite{Do} to transfer maps
\[\tau_{k,p} \colon H_*(\cC_{k}(M)) \lra H_*(\cC_{k-p}(M))\]
obtained by summing over all ways of deleting $p$ circles.
\end{proof}

\section{Generalizations and applications}\label{sec.generalizations} In this section we describe how to obtain generalizations of the previous result, and show that one of these implies Corollary \ref{corhomstabdiff}.

\subsection{Circles with additional structure}To define $\cC_k(M)$ we took the quotient of the space of ordinary unlinked embeddings of circles by the group $\mr{Diff}(\sqcup_k S^1) \cong \mr{Diff}(S^1) \wr \fS_k$ of all diffeomorphisms. One can generalize this by changing the words ``ordinary'' and ``all'' in the previous sentence. For example, one can replace $\mr{Diff}(S^1)$ by the subgroup $\mr{Diff}_+(S^1)$ of orientation preserving diffeomorphisms. 

The example of spaces of parametrized circles with a trivialization of their normal bundle of self-linking number $0$ will be important in the next subsection. Suppose $c \subset M$ is an oriented embedded circle and assume its normal bundle $\nu_c$ is orientable (it will be when $c$ is unlinked). Since a 2-dimensional oriented vector bundle over a circle is trivial, $\nu_c$ in particular admits a non-zero section $s$. The set of homotopy classes of non-zero sections is a torsor for $\pi_0(\mr{Map}(c,\mr{GL}^+_2(\bR))) \cong \bZ$.

Further suppose $c$ is null-homologous, so that it bounds a surface $L$. Given a section $s$ of the normal bundle $\nu_c$ of $c$, push $c$ in the direction $s$ to obtain the \emph{push-off} $p_{s}(c)$ of $c$. This is well-defined up to isotopy and without loss generality transverse to $L$. The \emph{self-linking number} $\mr{slk}_s(c)$ of $c$ is then defined to be the number of intersections of $p_s(c)$ with $L$, counted with sign. The number $\mr{slk}_s(c)$ depends only on the homotopy class of $s$, and if we act by $n \in \pi_0(\mr{Map}(c,\mr{GL}_2^+(\bR))) \cong \bZ$ the linking number also changes by $n$. Thus there is a canonical homotopy class $[s_0]$ of sections defined by demanding $\mr{slk}_{s_0}(c) = 0$. We say that an element in this homotopy class is \emph{of self-linking number zero}. Using the orientation of circle, there is up to homotopy a canonical second section of $\nu_c$, so that there is in fact a canonical homotopy class of trivializations of $\nu_c$ of self-linking number zero.

For $\bfc$ of $\mr{Emb}^{\mr{unl}}(\bigsqcup_k S^1,D^3)$, let $\mr{Triv}_{\mr{slk}=0}(\nu_{\bfc})$ be the product over $c \in \bfc$ of the spaces of trivializations of the normal bundle $\nu_{c}$ of $c \in \bfc$ of self-linking number zero. These spaces form a fiber bundle $\mr{Triv}_{\mr{slk}=0}(k,M)$ over $\mr{Emb}^{\mr{unl}}(\bigsqcup_k S^1,D^3)$ with an $\mathfrak S_k$-action extending that on unlinked embeddings. 

\begin{definition}The \emph{space of $k$ unlinked parametrized circles with trivializations of their normal bundles of self-linking number zero} is given by
\[\cC^\mr{triv}_k(M) \coloneqq \mr{Triv}_{\mr{slk}=0}(k,M)/\mathfrak S_k\]
\end{definition} 

As $\cC^\mr{triv}_k(-)$ is a continuous functor on the topological  category of 3-manifolds and embeddings, we can define a stabilization map for $\cC^\mr{triv}_k(M)$ like we did for $\cC_k(M)$ in the introduction.

\begin{corollary}\label{cor.normal} The map $t_* \colon H_*(\cC^\mr{triv}_k) \to H_*(\cC^\mr{triv}_{k+1})$ is always injective and an isomorphism for $* \leq \lfloor \frac{k-2}{2} \rfloor$.\end{corollary}

\begin{proof}The parametrizations of the circle and trivializations of the normal bundle can be carried along during most of the argument and one can use a similar semisimplicial resolution, with the slight modification that one should require $\eta|_{\partial D^2 \times \{1\}}$ has the same orientation as the circle it is attached to. Then in the statement of Proposition \ref{prop.sse1page} we pick up a term of the form $X^{p+1}$, with $X$ path-connected (this required the small modification to the resolution). This in turn requires a slight modification to the spectral sequence argument, as for the labeled configuration spaces in Section 6 of \cite{RW}.\end{proof}

\begin{remark}In general one can prove homological stability as long as the space of additional structures on each circle is either path-connected, or has two components corresponding to the two different orientations of the circle. The latter is the case in Corollary \ref{cor.normal}.\end{remark}

\subsection{Homological stability for diffeomorphisms} \label{subsec.diffeomorphisms} We now prove homological stability for diffeomorphisms of connected sums of a path-connected manifold $M$ with $k$ copies of $D^2 \times S^1$, permuting the torus boundary components and fixing pointwise the original boundary. This is the most restricted version of the diffeomorphism group for which one can prove homological stability. The group of diffeomorphisms of $M \#_k D^2 \times S^1$ fixing the entire boundary is a candidate for representation stability instead. See \cite{hanglamthesis} for other homological stability results for diffeomorphism groups of 3-manifolds.

Let $M \#_k D^2 \times S^1$ denote the connected sum of $M$ with $k$ copies of $D^2 \times S^1$. This has boundary $\partial M$ together with $k$ additional boundary components diffeomorphic to a torus $\bT^2$. For each of these we fix a diffeomorphism with $\bT^2$.

\begin{definition}Let $\mr{Diff}_{\partial M}(M \#_k D^2 \times S^1,\mathfrak S_\partial)$ be the topological group of diffeomorphisms of $M \#_k D^2 \times S^1$ which fix $\partial M$ pointwise and permute the boundary components preserving their identifications with $\bT^2$, with the $C^\infty$-topology.\end{definition}

\begin{lemma}\label{lem.cctrivident} We have that $\cC^\mr{triv}_k(M) \simeq \mr{Emb}^\mr{unl}(\sqcup_k D^2 \times S^1,M)/\fS_k$.
\end{lemma}

\begin{proof} There is a map $\mr{Emb}^\mr{unl}(\bigsqcup_k D^2 \times S^1,M)/\fS_k \to \cC^\mr{trivn}_k(M)$ given by restricting to $\bigsqcup_k \{0\} \times S^1$ and using the trivialization induced by the differential at $\{0\} \times S^1$. This is a weak equivalence by a similar argument as in Lemma \ref{lem.xthtoxtweq}.
\end{proof}

\begin{proof}[Proof of Corollary \ref{corhomstabdiff}] In light of Lemma \ref{lem.cctrivident} let us shorten $\mr{Emb}^\mr{unl}(\sqcup_k D^2 \times S^1,M)/\fS_k$ to $\cC^\mr{th,triv}_k(M)$. Isotopy extension implies that the action of $\mr{Diff}_\partial(M)$ on $\cC^\mr{th,triv}_k(M)$ induces a fiber sequence
	\[\cC^\mr{th,triv}_k(M) \to B\mr{Diff}_{\partial M}(M \#_k D^2 \times S^1,\mathfrak S_\partial) \to B\mr{Diff}_\partial(M)\]
The stabilization map on $\mr{Emb}^\mr{unl}(\sqcup_k D^2 \times S^1,M)/\fS_k$ induces one on $ B\mr{Diff}_{\partial M}(M \#_k D^2 \times S^1,\mathfrak S_\partial)$ such that following diagram commutes
\[\xymatrix{\cC^\mr{th,triv}_k(M) \ar[r]^-t \ar[d] & \cC^\mr{th,triv}_{k+1}(M) \ar[d] \\
 B\mr{Diff}_{\partial M}(M \#_k D^2 \times S^1,\mathfrak S_\partial)  \ar[r]^-{t'} \ar[d] &  B\mr{Diff}_{\partial M}(M \#_{k+1} D^2 \times S^1,\mathfrak S_\partial) \ar[d] \\
 B\mr{Diff}_\partial(M) \ar@{=}[r] &  B\mr{Diff}_{\partial}(M)
}\]
This gives rise to a relative Serre spectral sequence
\[\xymatrix{E^2_{pq} = H_p(B\mr{Diff}_\partial(M);H_q(\cC^\mr{th,triv}_{k+1}(M),\cC^\mr{th,triv}_k(M))) \ar@{=>}[d] \\ H_{p+q}(B\mr{Diff}_{\partial M}(M \#_{k+1} D^2 \times S^1,\mathfrak S_\partial),B\mr{Diff}_{\partial M}(M \#_{k} D^2 \times S^1,\mathfrak S_\partial))}\]
That $t'$ induces an isomorphism or surjection in the desired range follows from the fact that $t$ does by Corollary \ref{cor.normal} and Lemma \ref{lem.cctrivident}.
\end{proof}

\bibliographystyle{unsrt}
\bibliography{thesis}

\def\cprime{$'$}
\begin{thebibliography}{10}

\bibitem{palmerthesis}
Martin Palmer.
\newblock Configuration spaces and homological stability, 2012.

\bibitem{hatcherbrendle}
Tara~E. Brendle and Allen Hatcher.
\newblock Configuration spaces of rings and wickets.
\newblock {\em Comment. Math. Helv.}, 88(1):131--162, 2013.

\bibitem{hatcherwahl}
Allen Hatcher and Nathalie Wahl.
\newblock Stabilization for mapping class groups of 3-manifolds.
\newblock {\em Duke Math. J.}, (155):205--269, 2010.

\bibitem{wilsonstring}
Jennifer C.~H. Wilson.
\newblock Representation stability for the cohomology of the pure string motion
  groups.
\newblock {\em Algebr. Geom. Topol.}, 12(2):909--931, 2012.

\bibitem{wilsonstring2}
Jennifer C.~H. Wilson.
\newblock {$\rm FI_W$}-modules and stability criteria for representations of
  classical {W}eyl groups.
\newblock {\em J. Algebra}, 420:269--332, 2014.

\bibitem{griffindiagonal}
James~T. Griffin.
\newblock Diagonal complexes and the integral homology of the automorphism
  group of a free product.
\newblock {\em Proc. Lond. Math. Soc. (3)}, 106(5):1087--1120, 2013.

\bibitem{RW}
Oscar Randal-Williams.
\newblock Homological stability for unordered configuration spaces.
\newblock {\em Quarterly Journal of Mathematics}, (64 (1)):303--326, 2013.

\bibitem{gmtw}
S{\o}ren Galatius, Ulrike Tillmann, Ib~Madsen, and Michael Weiss.
\newblock The homotopy type of the cobordism category.
\newblock {\em Acta Math.}, 202(2):195--239, 2009.

\bibitem{M}
J.~P. May.
\newblock {\em The geometry of iterated loop spaces}.
\newblock Springer-Verlag, Berlin, 1972.
\newblock Lecture Notes in Mathematics, Vol. 271.

\bibitem{grwstab1}
Soren Galatius and Oscar Randal-Williams.
\newblock Homological stability for moduli spaces of high dimensional
  manifolds, {I}.
\newblock {\em preprint}, 2014.
\newblock \url{http://arxiv.org/abs/1403.2334}.

\bibitem{perlmutterstab}
Nathan Perlmutter.
\newblock Homological stability for diffeomorphism groups of high dimensional
  handlebodies.
\newblock {\em preprint}, 2015.
\newblock \url{http://arxiv.org/abs/1510.02571}.

\bibitem{hatcherautfn}
Allen Hatcher.
\newblock Homological stability for automorphism groups of free groups.
\newblock {\em Comment. Math. Helv.}, 70(1):39--62, 1995.

\bibitem{hanglamthesis}
Chor Hang~Lam.
\newblock Homological stability for diffeomorphism groups of 3-manifolds, 2015.

\bibitem{rolfsen}
Dale Rolfsen.
\newblock {\em Knots and links}.
\newblock Publish or Perish, Inc., Berkeley, Calif., 1976.
\newblock Mathematics Lecture Series, No. 7.

\bibitem{michor}
Peter~W. Michor.
\newblock {\em Manifolds of differentiable mappings}, volume~3 of {\em Shiva
  Mathematics Series}.
\newblock Shiva Publishing Ltd., Nantwich, 1980.

\bibitem{oscarresolutions}
Oscar Randal-Williams.
\newblock Resolutions of moduli spaces and homological stability.
\newblock {\em J. Eur. Math. Soc. (JEMS)}, 18(1):1--81, 2016.

\bibitem{grayson}
Matthew~A. Grayson.
\newblock Shortening embedded curves.
\newblock {\em Ann. of Math. (2)}, 129(1):71--111, 1989.

\bibitem{Do}
Albrecht Dold.
\newblock Decomposition theorems for {$S(n)$}-complexes.
\newblock {\em Ann. of Math. (2)}, 75:8--16, 1962.

\end{thebibliography}

\end{document}